\documentclass{svproc}
\usepackage{url}

\usepackage{mathrsfs}
\usepackage{amssymb}
\usepackage[arrow, matrix, curve, xdvi, dvips]{xy}
\usepackage{tikz}
\usepackage{tikz-cd}
\newcommand{\Sp}{\rm{Sp}}
\newcommand{\U}{\rm{U}}

\begin{document}
\mainmatter              
\title{Sub-Riemannian geodesics on nested principal bundles}
\titlerunning{Geodesics on nested principal bundles}  
%
\author{Mauricio Godoy Molina\inst{1}\thanks{Partially supported by grants Fondecyt \#1181084 by the Chilean Research Council and DI20-0023 by Universidad de La Frontera.} \and Irina Markina\inst{2}\thanks{Partially supported by grant \# 262363/O70 of the Norwegian Research Council.} 
}
\authorrunning{Godoy Molina, Markina 
} 
%
\tocauthor{Godoy Molina, Markina}
\institute{Departamento de Matem\'atica y Estad\'istica\\Universidad de La Frontera\\Temuco, Chile\\
\email{mauricio.godoy@ufrontera.cl}
\and
Department of Mathematics\\University of Bergen\\
Bergen, Norway\\
\email{irina.markina@uib.no}}

\maketitle              

\begin{abstract}
{ We study the interplay between geodesics on two non-holono\-mic systems that are related by the action of a Lie group on them. After some geometric preliminaries, we use the Hamiltonian formalism to write the parametric form of geodesics. We present several geometric examples, including a non-holonomic structure on the Gromoll-Meyer exotic sphere and twistor space.}
\end{abstract}
\section{Introduction}

{Our paper is related to the geometric control theory of mechanical systems with symmetries. To be precise, we consider a configuration space $M$ together with a Lie group $H$ acting on $M$ which preserves some constraints on the velocities. Of particular importance are non-holonomic constraints, which are restrictions that can not be reduced to position constraints. In our model these restrictions are modelled as a smooth distribution $D$ inside the tangent bundle $TM$ of the configuration space which is transverse to the infinitesimal action of $H$. All these data are combined in a geometric structure called {\it principal bundle}. We also assume that there exists a Lie subgroup $K<H$ such that the restriction of the action of $K$ is also a principal bundle. This leads to an interaction of two non-holonomic systems. By making use of the Hamiltonian formalism we study the interplay of the geodesic curves in these non-holonomic systems. Geometric examples of this construction include the quaternionic Hopf fibration \cite{BFI}, the Gromoll-Meyer exotic sphere \cite{Gromoll} and the twistor bundle of $S^4$ \cite{AHS}.
}

\section{Nested principal bundles}

In this paper, all manifolds and Lie groups are assumed to be connected.


\begin{definition}\label{def:ppal}
Let $H$ be a Lie group. A submersion $\pi_H\colon M\to N$ is a principal $H$-bundle if $H$ acts freely and transitively from the right on the fibers $\pi_H^{-1}(n)$, $n\in N$. 
\end{definition}

We denote a principal $H$-bundle as $H\curvearrowright M\stackrel{\pi_H}{\longrightarrow}N$. Note that $N$ is diffeomorphic to the quotient $M/H$. The {\it vertical bundle} ${\mathscr{V}}\to M$ is a vector bundle defined by ${\mathscr{V}}_m=\ker d_m\pi_H$, $m\in M$.
Let ${\mathfrak{h}}$ be the Lie algebra of the group $H$.
Given a vector $\xi\in{\mathfrak{h}}$, we define the {\it fundamental vector field} on $M$ by
\[
\sigma_m(\xi)=\left.\frac{d}{dt}\right|_{t=0}m.\exp_H(t\xi),\quad m\in M,
\]
where $\exp_H\colon{\mathfrak{h}}\to H$ is the group exponential, and $m.h$ denotes the action of $h\in H$ on $m\in M$. 
Let $K<H$ be a closed Lie subgroup of $H$ with the Lie algebra $\mathfrak k$. We say that the triplet $(M,H,K)$ is a {\it nested principal bundle} if the restriction of the action to $K$ is also a principal bundle. In this case, if $H\curvearrowright M\stackrel{\pi_H}{\longrightarrow}M/H$ and $K\curvearrowright M\stackrel{\pi_K}{\longrightarrow}M/K$ are the principal $H$- and $K$-bundles respectively, we have the vertical bundles
\[
{\mathscr{V}}_H=\ker d\pi_H\cong\mathfrak{h}\times M\quad\mbox{and}\quad{\mathscr{V}}_K=\ker d\pi_K\cong\mathfrak{k}\times M.
\]
Consider two Ehresmann connections ${\mathscr{D}}_H\hookrightarrow TM$ and ${\mathscr{D}}_K\hookrightarrow TM$ for $\pi_H$ and $\pi_K$ respectively, that is
$
\ker d\pi_H\oplus{\mathscr{D}}_H=\ker d\pi_K\oplus{\mathscr{D}}_K=TM$.
Assume that the distributions ${\mathscr{D}}_H$ and ${\mathscr{D}}_K$ are invariant under the action of $H$ and $K$ respectively. The aim of the present paper is to study the sub-Riemannian structure of the triplet $(M/K,{\mathscr{D}},g_{\mathscr{D}})$, where the distribution ${\mathscr{D}}\hookrightarrow T(M/K)$ is defined by ${\mathscr{D}}=d\pi_K({\mathscr{D}}_H)$ and the metric $g_{\mathscr{D}}$ will be defined later. 

Observe that $\pi\colon M/K\to M/H$ is a submersion where a fiber is the homogeneous space $H/K$, so that we have the following diagram
\begin{equation}\label{eq:maindiag}
\xymatrix{&H\ar@/^/[d]\ar@{->}[r]&H/K\ar@{->}[d]\\K\ar@{^{(}->}[ur]\ar@/^/[r]&M\ar@{->}[r]^{\pi_K}\ar@{->}[d]_{\pi_H}&M/K\ar@{->}[dl]^{\pi}\\&M/H&}
\end{equation}
where the two triangles commute. Notice that, in principle, the space $H/K$ is just a homogeneous space and $H/K\to M/K\to M/H$ is just a fibration.
\begin{lemma}\label{lemma:Ehr}
The distribution ${\mathscr{D}}$ is an Ehresmann connection for~$\pi$.
\end{lemma}

\begin{proof}
By a dimension counting argument, it is enough to show that ${\mathscr{D}}$ is transverse to $\ker d\pi$. If $v\in{\mathscr{D}}\cap\ker d\pi$, then $v=d\pi_{K}(w)$ for some $w\in{\mathscr{D}}_H$, and $d\pi(v)=0$. This implies that
$
d\pi(d\pi_K(w))=d(\pi\circ\pi_K)(w)=0$,
by the chain rule. Since diagram (\ref{eq:maindiag}) commutes, we know that $\pi\circ\pi_K=\pi_H$, and thus $d\pi_H(w)=0$. From this, we conclude that $w\in{\mathscr{D}}_H\cap\ker d\pi_H$. By the assumption, ${\mathscr{D}}_H$ is an Ehresmann connection for $\pi_H$, and therefore $w=0$. This implies that $v=0$.\qed
\end{proof}

\begin{remark}\label{rem:iso}
Note that $d\pi_K$ gives the isomorphisms  ${\mathscr{D}}_H\cong {\mathscr{D}}$ and ${\mathscr{D}}_K\cong T(M/K)$.
It follows that the vector $w$ such that $v=d\pi_K(w)$ is unique for any $v\in{\mathscr{D}}$.
\end{remark}


\section{Hamiltonians in nested principal bundles}\label{sec:ham}


Let $g$ be a Riemannian metric on $M$, such that ${\mathscr{D}}_H={\mathscr{V}}_H^\bot$ and ${\mathscr{D}}_K={\mathscr{V}}_K^\bot$ with respect to $g$, and such that $g$ is invariant under the action of $H$ (thus also invariant under $K$). It follows that ${\mathscr{D}}_H$ is $H$-invariant and ${\mathscr{D}}_K$ is $K$-invariant. 

For each $m\in M$, the restriction $g|_{{\mathscr{V}}_H}$ defines the positive definite symmetric bilinear form on the Lie algebra ${\mathfrak{h}}$
\[
{\mathbb{I}}^H_m(\xi,\eta)=g|_{{\mathscr{V}}_H}\big(\sigma_m(\xi),\sigma_m(\eta)\big),\quad m\in M,\quad\xi,\eta\in{\mathfrak{h}}.
\]
We require that ${\mathbb{I}}^H_m$ does not depend on $m\in M$. 
According to the terminology in \cite[Chapter 11]{M}, the metric $g$ satisfying all of these hypotheses is called of constant bi-invariant type 
%
with respect to both group actions. The metrics $g_{{\mathscr{D}}_K}$ and $g_{{\mathscr{D}}_H}$, obtained by restriction, are sub-Riemannian metrics on $M$ for the distributions ${\mathscr{D}}_K$ and ${\mathscr{D}}_H$, respectively. 

Define the Riemannian metrics $g_{M/K}$ and $g_{M/H}$ by the equalities
\begin{eqnarray}
g(v,w)=&g_{M/K}(d_m\pi_K(v),d_m\pi_K(w)),\quad v,w\in{\mathscr{D}}_K,\quad m\in M,\label{eq:MK}\\
g(v,w)=&g_{M/H}(d_m\pi_H(v),d_m\pi_H(w)),\quad v,w\in{\mathscr{D}}_H,\quad m\in M.\label{eq:MH}
\end{eqnarray}

\begin{proposition}
The map $\pi\colon M/K\to M/H$ is a Riemannian submersion, with respect to the metrics (\ref{eq:MK}) and (\ref{eq:MH}), respectively.
\end{proposition}


\begin{remark}
The map $d\pi_K\vert_{{\mathscr{D}}_K}\colon {\mathscr{D}}_K\to T(M/K)$
is an isometry on each fiber with respect to the Riemannian metrics $g_{{\mathscr{D}}_K}$ and $g_{M/K}$.
\end{remark}

Given a smooth subbdundle $D$ of $TM$ and a metric tensor $g_D$ on $M$ defined only for vectors belonging to $D_m$, $m\in M$, the {\it $g_D$ sharp} map $\sharp^{g_D}\colon T^*M\to TM$, is the unique map satisfying ${\rm im}\,\sharp^{g_D}=D$ and if $\lambda\in T^*_mM$, then $\sharp^{g_D}(\lambda)\in D_m$ is the unique vector for which
$
\lambda(w)=g_D(\sharp^{g_D}(\lambda),w)$, for all $w\in D_m$, $m\in M$.
The cometric $g_D^*\colon T^*M\times T^*M\to{\mathbb R}$ {\em associated to} $g_D$ is defined by $g_D^*(\lambda,\mu)=g_D\big(\sharp^{g_D}(\lambda),\sharp^{g_D}(\mu)\big)$. Every cometric 
defines a function ${\mathbf H}\in C^\infty(T^*M)$, called the {\it Hamiltonian} associated to $g^*$, by the formula
\[
{\mathbf H}(m,\lambda)=\frac12g^*(\lambda,\lambda),\quad\lambda\in T^*_mM.
\]

Using the metrics $g$, $g_{{\mathscr{D}}_H}$, $g_{{\mathscr{D}}_K}$, $g_{{\mathscr{V}}_H}=g|_{{\mathscr{V}}_H}$ and $g_{{\mathscr{V}}_K}=g|_{{\mathscr{V}}_K}$ on $M$, we define the respective Hamiltonian functions on $T^*M$. Note that $g=g_{{\mathscr{D}}_H}+g_{{\mathscr{V}}_H}=g_{{\mathscr{D}}_K}+g_{{\mathscr{V}}_K}$. This implies the equalities  ${\mathbf H}_M={\mathbf H}^{{\mathscr{D}}_H}+{\mathbf H}^{{\mathscr{V}}_H}={\mathbf H}^{{\mathscr{D}}_K}+{\mathbf H}^{{\mathscr{V}}_K}$.

Similarly, considering the metrics $g_{M/K}$, $g_{\mathscr{D}}=g_{M/K}|_{\mathscr{D}}$ and $g_{\mathscr{V}}=g_{M/K}|_{\mathscr{V}}$ on $M/K$, one has the respective Hamiltonian functions on $T^*(M/K)$ and a decomposition ${\mathbf H}_{M/K}={\mathbf H}^{\mathscr{D}}+{\mathbf H}^{\mathscr{V}}$.

Let us denote by $\pi^*_K\colon T^*(M/K)\to T^*M$ the induced map of cotangent bundles defined by
$
\pi_K^*(\lambda)(v)=\lambda(d_m\pi_K(v))$, for $\lambda\in T^*(M/K)$, $v\in T_mM$, $m\in M$.

\begin{proposition}\label{prop:Hamcomm}
The following identities take place:
\begin{enumerate}
\item[(a)] ${\mathbf H}^{{\mathscr D}_H}\circ\pi^*_K={\mathbf H}^{\mathscr D}$,
\item[(b)] $\big({\mathbf H}^{{\mathscr D}_K}-{\mathbf H}^{{\mathscr D}_H}\big)\circ\pi^*_K={\mathbf H}^{\mathscr V}$.
\item[(c)] As a consequence, we have that $\big({\mathbf H}_M-{\mathbf H}^{{\mathscr V}_K}\big)\circ\pi^*_K={\mathbf H}^{{\mathscr D}_K}\circ\pi^*_K={\mathbf H}_{M/K}$.
\end{enumerate} 
\end{proposition}

\begin{proof}
Let $n=\dim M$, $r={\rm rk}\,{\mathscr{D}}_H$ and $s={\rm rk}\,{\mathscr{D}}_K$. 
Consider $X_1,\ldots,X_r,\ldots,X_s$, $V_1,\ldots,V_{n-s}$ a local frame of vector fields of ${\mathfrak{X}}(M)$ orthonormal with respect to $g$, where the vector fields $X_1,\ldots,X_r$ span the local sections of ${\mathscr{D}}_H$, the vector fields $X_1,\ldots,X_s$ span the local sections of ${\mathscr{D}}_K$ and the vector fields $V_1,\ldots,V_{n-s}$ span the local sections of ${\mathscr{V}}_K$. Denote by $Y_j=d\pi_K(X_j)$, for $j=1,\ldots,s$. It follows from this choice that $Y_1,\ldots,Y_s$ are orthonormal with respect to $g_{M/K}$ and that $Y_1,\ldots,Y_r$ span ${\mathscr{D}}$ at each $p\in M/K$.

Given the Riemannian metric $g$, we have a canonical isomorphism between $TM$ and $T^*M$, thus we have the dual frames $p_{X_1},\ldots,p_{X_s},p_{V_1},\ldots,p_{V_{n-s}}$ defined on $M$ and $p_{Y_1},\ldots,p_{Y_s}$ defined on $M/K$. With all of these notations, we have the Riemannian Hamiltonians
\[
{\mathbf H}_M(\lambda)=\frac12\sum_{i=1}^sg^*(\lambda,p_{X_i})^2+\frac12\sum_{j=1}^{n-s}g^*(\lambda,p_{V_j})^2,\;
{\mathbf H}_{M/K}(\mu)=\frac12\sum_{i=1}^sg_{M/K}^*(\mu,p_{Y_i})^2,
\]
where $\lambda\in T^*M$ and $\mu\in T^*(M/K)$, the horizontal Hamiltonians
\[
{\mathbf H}^{{\mathscr D}_H}(\lambda)=\frac12\sum_{i=1}^rg^*(\lambda,p_{X_i})^2,\quad
{\mathbf H}^{{\mathscr D}_K}(\lambda)=\frac12\sum_{i=1}^sg^*(\lambda,p_{X_i})^2,\;
\]
\[
{\mathbf H}^{\mathscr D}(\mu)=\frac12\sum_{i=1}^rg_{M/K}^*(\mu,p_{Y_i})^2,
\]
and the vertical Hamiltonians
\[
{\mathbf H}^{{\mathscr V}_K}(\lambda)=\frac12\sum_{j=1}^{n-s}g^*(\lambda,p_{V_j})^2,\quad\quad
{\mathbf H}^{\mathscr V}(\mu)=\frac12\sum_{j=r+1}^{s}g_{M/K}^*(\mu,p_{Y_j})^2.
\]

Before we start computing, it is convenient to note that the diagram
\begin{equation}\label{diag:sharppi}
\xymatrix{T^*(M/K)\ar@{->}[rr]^{\pi_K^*}\ar@{->}[d]_{\sharp^{g_{M/K}}}&&T^*M\ar@{->}[d]^{\sharp^g}\\T(M/K)\ar@{<-}[rr]^{d\pi_K}&&TM}
\end{equation}
commutes, that is, we have the equality $d\pi_K\circ\sharp^g\circ\pi_K^*=\sharp^{g_{M/K}}$. To see this,  observe that $d\pi_K(X_j)=Y_j$ implies that $\pi_K^*(p_{Y_j})=p_{X_j}$. Indeed
\[
\pi_K^*(p_{Y_j})(X_i)=p_{Y_j}(d\pi_K(X_i))=p_{Y_j}(Y_i)=\delta_{ji},
\]
where $\delta_{ji}$ is the Kronecker delta. Also note that $\sharp^gp_{X_i}=X_i$, which implies that $\sharp^g\big(\pi_K^*(p_{Y_j})\big)=X_j$. Finally, we can conclude that
$
d\pi_K\Big(\sharp^g\big(\pi_K^*(p_{Y_j})\big)\Big)=d\pi_K(X_j)=Y_j=\sharp^{g_{M/K}}p_{Y_j}
$.
By linearity, the commutativity of the diagram follows.

To prove {\it (a)}, let $\mu\in T^*(M/K)$, then we compute
\begin{eqnarray*}
\big({\mathbf H}^{{\mathscr D}_H}\circ\pi^*_K\big)(\mu)&=&\frac12\sum_{i=1}^rg^*(\pi_K^*(\mu),p_{X_i})^2=\frac12\sum_{i=1}^rg(\sharp^g\pi_K^*(\mu),\sharp^gp_{X_i})^2\\
&=&\frac12\sum_{i=1}^rg(\sharp^g\pi_K^*(\mu),X_i)^2=\frac12\sum_{i=1}^rg_{M/K}(d\pi_K\sharp^g\pi_K^*(\mu),d\pi_KX_i)^2\\
&=&\frac12\sum_{i=1}^rg_{M/K}(\sharp^{g_{M/K}}\mu,Y_i)^2=\frac12\sum_{i=1}^rg_{M/K}^*(\mu,p_{Y_i})^2
={\mathbf H}^{{{\mathscr D}}}(\mu)
\end{eqnarray*}
In the fourth equality we used the fact that $\pi_K$ is a Riemannian submersion, and in the fifth one, the commutativity of diagram (\ref{diag:sharppi}).

A similar computation can be performed for {\it (b)}. Equality {\it (c)} can be obtained adding {\it (a)} and {\it (b)}.\qed
\end{proof}

\begin{remark}
Proposition~\ref{prop:Hamcomm} is a special case of the so-called lifted Hamiltonian in \cite{Grong}. In our case, the map $\pi^2$ defined in \cite{Grong} corresponds to $(\pi_K^*|_{{\rm im}\,\pi_K^*})^{-1}$ and extended by zero to the orthogonal complement of ${\rm im}\,\pi_K^*$. We observe that since $TM={\mathscr{D}}_K\oplus_\bot{\mathscr{V}}_K$, then the metric produces the two isomorphisms ${\mathscr{D}}_K^*\cong{\rm Ann}({\mathscr{V}}_K)$ and ${\mathscr{V}}_K^*\cong{\rm Ann}({\mathscr{D}}_K)$. {Here ${\rm Ann}(E)\subset T^*M$ is the annihilator of the vector subbundle $E\subset TM$.}
\end{remark}


\section{Sub-Riemannian geodesics}\label{sec:geod}

\begin{definition}\label{def:sR}
A sub-Riemannian manifold is a triplet $(M,D,g_D)$, where $D\hookrightarrow TM$ is a smooth (integrable/non-integrable) vector subbundle of $TM$ and $g_D$ is a metric tensor on $M$ defined only for vectors belonging to $D_p$ for all $p\in M$. 
\end{definition}
%
Given a Riemannian metric $g=g_{TM}$ on $M$, we denote by ${\mathbf{H}}_M$ the Hamiltonian associated to $g^*$. Given a sub-Riemannian metric $g_D$ on $M$, we denote by ${\mathbf{H}}^D$ the Hamiltonian associated to $g_D^*$.

\begin{definition}
Let $(M,D,g_D)$ be a sub-Riemannian manifold. The image of the projection $\Pi_M\colon T^*M\to M$ of the flow $e^{t\overrightarrow{\mathbf{H}}^D}$ of the Hamiltonian vector field $\overrightarrow{\mathbf{H}}^D$ associated to ${\mathbf{H}}^D$ is called a  sub-Riemannian geodesic.
\end{definition}

%
%

The aim of this section is to relate the sub-Riemannian geodesics in the sub-Riemannian manifolds $(M,{\mathscr D}_H,g_{{\mathscr D}_H})$ and $(M/K,{\mathscr D},g_{{\mathscr D}})$. Notice that we have the following commutative diagram of cotangent bundles
\[
\xymatrix{&T^*M\ar@{->}[dd]^{\Pi_M}&T^*(M/K)\ar@{->}[l]_{\pi_K^*}\ar@{->}[dd]^{\Pi_{M/K}}\\T^*(M/H)\ar@{->}[ur]^{\pi_H^*}\ar@{->}[urr]_{\!\!\!\!\!\!\!\!\pi^*}\ar@{->}[dd]_{\Pi_{M/H}}&&\\&M\ar@{->}[r]^{\pi_K}\ar@{->}[dl]_{\pi_H}&M/K\ar@{->}[dll]^{\pi}\\M/H&&}
\]

Let $(m,\lambda)\in T^*M$, then the sub-Riemannian geodesic starting at $m\in M$ with covector $\lambda\in T^*_mM$ and tangent to ${\mathscr D}_H$ is given by
\[
\gamma_{{\mathscr D}_H}^{sR}(t;m,\lambda)=\big(\Pi_M\circ e^{t{\overrightarrow{\mathbf{H}}^{{\mathscr D}_H}}}\big)(m,\lambda),
\]
for $t>0$ sufficiently small. An analogous definition is valid for  sub-Riemannian geodesics $\gamma_{{\mathscr D}}^{sR}(t;n,\mu)$, $(n,\mu)\in T^*(M/K)$.  Let $\omega_M$ and $\omega_{M/K}$ be the canonical symplectic forms on $T^*M$ and $T^*(M/K)$, respectively.



\begin{theorem}\label{th:geod}
Let $(n,\mu)\in T^*(M/K)$, and consider $m\in M$ with $n=\pi_K(m)$ and $\pi_K^*\mu=\lambda\in T^*_mM$. If the map $\pi_K^*\colon T^*(M/K)\to T^*M$ is a symplectomorphism, then
\begin{equation}\label{eq:geod}
\pi_K\big(\gamma_{{\mathscr D}_H}^{sR}(t;m,\lambda)\big)=\gamma_{{\mathscr D}}^{sR}(t;n,\mu).
\end{equation}
\end{theorem}

\begin{proof}
For any $w\in TT^*(M/K)$, we have
\begin{eqnarray*}
\omega_{M/K}(\overrightarrow{\mathbf{H}}^{\mathscr D},w)&=&d{\mathbf{H}}^{\mathscr{D}}(w)=d\big({\mathbf{H}}^{{\mathscr{D}}_H}\circ\pi_K^*\big)(w)=d{\mathbf{H}}^{{\mathscr{D}}_H}\big(d\pi_K^*(w)\big)\\
&=&\omega_M\big(\overrightarrow{\mathbf{H}}^{{\mathscr D}_H},d\pi_K^*(w)\big),
\end{eqnarray*}
from the definition of the Hamiltonian vector fields, the chain rule and Proposition~\ref{prop:Hamcomm}. The map $\pi_K^*\colon T^*(M/K)\to T^*M$ is a symplectomorphism, that is
$
\omega_{M/K}(\alpha,\beta)=\omega_M\big(d\pi_K^*(\alpha),d\pi_K^*(\beta)\big)$,
for all $\alpha,\beta\in TT^*M/K$. Then for any $w\in TT^*(M/K)$ we also have
$
\omega_{M/K}(\vec{\mathbf{H}}^{\mathscr{D}},w)=\omega_M\big(d\pi_K^*(\vec{\mathbf{H}}^{\mathscr{D}}),d\pi_K^*(w)\big)$.
Since the symplectic form $\omega_M$ is non-degenerate, we deduce that $\overrightarrow{\mathbf{H}}^{\mathscr D_H}=d\pi_K^*(\overrightarrow{\mathbf{H}}^{\mathscr D})$ from
$$
\omega_M\big(d\pi_K^*(\overrightarrow{\mathbf{H}}^{\mathscr D}),d\pi_K^*(w)\big)=\omega_M\big(\overrightarrow{\mathbf{H}}^{\mathscr D_H},d\pi_K^*(w)\big).
$$ 
This implies that the flows of the Hamiltonian vector fields are related by
\[
e^{t\overrightarrow{\mathbf{H}}^{\mathscr D_H}}(m,\lambda)=e^{td\pi_K^*\big(\overrightarrow{\mathbf{H}}^{\mathscr D}\big)}(m,\lambda)=\pi_K^*\circ e^{t\overrightarrow{\mathbf{H}}^{\mathscr D}}(n,\mu).
\]

To complete the proof, observe that the natural diagram
\begin{equation}\label{diag:sharppi}
\xymatrix{T^*M\ar@{<-}[rr]^{\pi_K^*}\ar@{->}[d]_{\Pi_M}&&T^*(M/K)\ar@{->}[d]^{\Pi_{M/K}}\\M\ar@{->}[rr]^{\pi_K}&&M/K}
\end{equation}
commutes, therefore
\begin{eqnarray*}
\pi_K\big(\gamma_{{\mathscr D}_H}^{sR}(t;m,\lambda)\big)&=&\big(\pi_K\circ\Pi_M\circ e^{t\overrightarrow{\mathbf{H}}^{\mathscr D_H}}\big)(t;m,\lambda)=\big(\pi_K\circ\Pi_M\circ\pi_K^*\circ e^{t\overrightarrow{\mathbf{H}}^{\mathscr D}}\big)(t;n,\mu)\\
&=&\big(\Pi_{M/K}\circ e^{t\overrightarrow{\mathbf{H}}^{\mathscr D}}\big)(t;n,\mu)=\gamma_{{\mathscr D}}^{sR}(t;n,\mu),
\end{eqnarray*}
which is the equality sought after.\qed
\end{proof}


%

\begin{corollary}
The sub-Riemannian geodesics $\gamma_{{\mathscr D}_H}^{sR}(t;m,\lambda)$ and $\gamma_{{\mathscr D}}^{sR}(t;n,\mu)$ in Theorem \ref{th:geod} have the same projection to $M/H$, that is
\[
\pi_H\big(\gamma_{{\mathscr D}_H}^{sR}(t;m,\lambda)\big)=\pi\big(\gamma_{{\mathscr D}}^{sR}(t;n,\mu)\big).
\]
\end{corollary}

%
%

\section{Examples}\label{sec:ex}

%
%
\subsection{The quaternionic Hopf fibration from $\Sp(2)$}\label{ssec:quatH}

An important special case of the case in which $K\triangleleft H$ is a normal subgroup is the quaternionic Hopf fibration, as constructed in~\cite{BFI}. A comprehensive introduction to Hopf fibrations, one can find in~\cite{GWZ}. 
Recall that the 10-dimensional compact symplectic group ${\Sp}(2)={\U}(4)\,\cap\,{\Sp}(4,\mathbb C)$ can be defined through quaternionic matrices as follows
\[
{\Sp}(2)=\left\{Q=\left(\begin{array}{cc}a&b\\c&d\end{array}\right)\in M(2\times 2,{\mathbb H})\colon Q^*Q=QQ^*={\rm id}\right\},
\]
where $Q^*$ denotes the transpose (quaternion) conjugate of $Q$. 
%
Let us consider the 6-dimensional subgroup
$
H=\left\{\left(\begin{array}{cc}
\mu & 0  \\
0 & \nu 
\end{array}\right)\in{\Sp}(2)\right\}\simeq{\Sp}(1)\times{\Sp}(1)
$
of diagonal quaternionic matrices in ${\Sp}(2)$. The right multiplication
\begin{equation}\label{eq:rightS7}
\left(\left(\begin{array}{cc}
\mu & 0  \\
0 & \nu 
\end{array}\right),Q\right)\mapsto Q \left( \begin{array}{cc}
\bar{\mu} & 0  \\
0 & \bar {\nu} 
\end{array} \right),
\end{equation}
defines a left group action of $H$ on $\Sp(2)$. 
The homogeneous space ${\Sp}(2)/H$ of this action is diffeomorphic to the usual 4-dimensional sphere $S^4$ by means of the ``stereographic projection'' 
\[
{\Sp}(2)/H\stackrel{\pi_H}{\longrightarrow}S^4,\quad H\left( \begin{array}{cc}
a & b  \\
c & d 
\end{array} \right)\mapsto (2 d\bar{b},|b|^2-|d|^2)=(-2c\bar{a},|c|^2-|a|^2),
\]
where $S^4=\{(q,x)\in \mathbb H\times \mathbb R\colon |q|^2+x^2=1\}$.

Let $K$ be the 3-dimensional subgroup of $H$ such that $\nu=1$. Restricting the left action (\ref{eq:rightS7}) to the subgroup $K$ determines the homogeneous space ${\Sp}(2)/K$ which is diffeomorphic to the usual 7-dimensional sphere $S^7$ by means of the projection map
$
K\left( \begin{array}{cc}
a&b\\ c&d 
\end{array} \right) \stackrel{\pi_K}{\longrightarrow} (b,d)
$,
where $S^7=\{(b,d)\in \mathbb H\times \mathbb H\colon |b|^2+|d|^2=1\}$. In this way, we have two maps as in the diagram
\[
\begin{tikzcd}
& {\Sp}(2) \arrow[dl, "\pi_K" description] \arrow[dr, ,  "\pi_H" description] & \\
   S^7 &    & S^4
 \end{tikzcd}
\]

As a direct consequence of this definition, we see that the quaternionic projective line $\mathbb H\mathbb P^1$ is diffeomorphic to the sphere $S^4$ under the map 
\[
[b:d]\mapsto (2 d\bar{b},|b|^2-|d|^2).
\]
Obviously this diffeomorphism is invariant under right multiplication by an element in ${\Sp}(1)$. Using these identifications, the projectivization map
\[
\mathbb H^2-\{(0,0)\}\to\mathbb H\mathbb P^1,\quad\quad (b,d)\mapsto[b:d]
\]
induces a map $h\colon S^7\to S^4$ called the quaternionic Hopf map. Since ${\Sp}(1)$ is diffeomorphic to the 3-dimensional sphere $S^3$, we have the diagram
\[
\begin{tikzcd}
&  & {\Sp}(2) \arrow[dl, "\pi_K" description] \arrow[dr, ,  "\pi_H" description] & \\
S^3 \arrow[r, bend left=20 ] &   S^7 \arrow[rr, "h"]   &    & S^4.
 \end{tikzcd}
\]
Observe that the quaternionic Hopf map provides a principal bundle with a typical fiber $S^3$, called the quaternionic Hopf fibration. 

The map $h$ corresponds to the submersion $\pi\colon{\Sp}(2)/K\to{\Sp}(2)/H$. It is known \cite{BFI,GM} that the distribution ${\mathscr{D}}=d\pi_K({\mathscr{D}}_H)$ on $S^7$ is bracket generating of step 2.
We endow ${\Sp}(2)$ with a bi-invariant Riemannian metric $g$ defined by
\[
g( u,v)={\rm{Re}}\,{\rm tr}(u\cdot v^*),\quad u,v\in{\mathfrak{sp}}(2)\subset M(2,{\mathbb H}),
\]
where $v^*$ denotes the transposed conjugate of $v$. The Ehresmann connections ${\mathscr{D}}_H$ and ${\mathscr{D}}_K$ are chosen as the left-translations of the orthogonal complements to ${\mathfrak h}\cong{\mathfrak{sp}}(1)\times{\mathfrak{sp}}(1)$ and ${\mathfrak k}\cong{\mathfrak{sp}}(1)\times\{0\}$ in ${\mathfrak{sp}}(2)$, with respect to $g$.

\subsection{Gromoll-Meyer exotic sphere}

The Gromoll-Meyer sphere \cite{Gromoll} is constructed in a similar fashion as the quaternionic Hopf fibration, but does not fit in the scheme of Subsection~\ref{ssec:quatH}, in the relation of sub-Riemannian geodesics, see~\cite{BFI}. 

Consider $M={\rm Sp}(2)$, and the subgroup of $M$:
$H={\rm Sp}(1)\times{\rm Sp}(1)$,
acting on the right by
\begin{equation}\label{eq:actionGM}
\left(\begin{array}{cc}x&y\\z&w\end{array}\right).(\lambda,\mu)=
\left(\begin{array}{cc}\bar\lambda x\mu&\bar\lambda y\\\bar\lambda z\mu&\bar\lambda w\end{array}\right)=
\left(\begin{array}{cc}\bar\lambda&0\\0&\bar\lambda\end{array}\right)\left(\begin{array}{cc}x&y\\z&w\end{array}\right)\left(\begin{array}{cc}\mu&0\\0&1\end{array}\right),
\end{equation}
where $(\lambda,\mu)\in H$ and $\left(\begin{array}{cc}x&y\\z&w\end{array}\right)\in M$. Consider the restriction of the action~(\ref{eq:actionGM}) to the subgroup $\Delta=\{(\lambda,\lambda)\in{\rm Sp}(1)\times{\rm Sp}(1)\}< H$, which is not normal in $H$.

As before, the maps $\pi_H$ and $\pi_\Delta$ are the quotient maps with respect to the action of $H$ and $\Delta$ respectively. In a similar way as before, it can be shown that the homogeneous space $M/H$ is diffeomorphic to the sphere $S^4$ with respect to the action~(\ref{eq:actionGM}). On the other hand, the homogeneous space $\Sigma_{GM}:=M/\Delta$, called the Gromoll-Meyer exotic sphere, is a seven dimensional manifold homeomorphic, but not diffeomorphic, to the sphere $S^7$, see \cite{Milnor56}. The corresponding submersion $\pi\colon\Sigma_{GM}\to S^4$ is an $S^3$-bundle over $S^4$ which is not a principal bundle. The distribution ${\mathscr{D}}=d\pi_\Delta({\mathscr{D}}_H)$ on $\Sigma_{GM}$ has been recently shown to be bracket generating of step 2, see~\cite{BFI}.

Endowing $M$ with the Riemannian metric $g$ from Subsection~\ref{ssec:quatH}, we define the Ehresmann connections ${\mathscr{D}}_H$ and ${\mathscr{D}}_\Delta$ as the orthogonal complements to ${\mathscr{V}}_H=\ker d\pi_H$ and ${\mathscr{V}}_K=\ker d\pi_\Delta$, with respect to $g$. In this case, the bilinear form ${\mathbb{I}}^H_m$ does depend on $m\in M$, therefore it is necessary to consider more general formulas for sub-Riemannian geodesics, see \cite{GG}.

\subsection{Twistor space of $S^4$}

Let $N$ be a four dimensional Riemannian manifold. The twistor space ${\mathbb T}(N)$ of $N$ is the fiber bundle of almost complex structures on $N$ that are compatible with the Riemannian metric. In the case of $N=S^4$ this yields to a well known construction where ${\mathbb T}(N)={\mathbb C}P^3$ and the bundle map is given by
\[
{\mathbb C}P^3\ni e\stackrel{T}{\longmapsto} e{\mathbb H}=e\oplus ej\in{\mathbb H}P^1\cong S^4,
\]
where $e\in{\mathbb C}P^3$ is thought of as a line in ${\mathbb C}^4\cong{\mathbb H}^2$. The fibers correspond to spheres ${\mathbb C}P^1\cong S^2$ endowed with its unique complex structure. The map $T$ is sometimes referred as the twistor projection. For more details, see~\cite{AHS,H}.

Consider the inclusion $S^1\hookrightarrow S^3\subset{\mathbb C}^2$ given by $e^{i\theta}\mapsto(e^{i\theta},0)$. The twistor projection $T$ fits in the following diagram
\[
\xymatrix{&S^3\ar@/^/[d]\ar@{->}[rr]^{H_1}&&{\mathbb C}P^1\ar@{->}[d]\\S^1\ar@{^{(}->}[ur]\ar@/^/[r]&S^7\ar@{->}[rr]^{H_3}\ar@{->}[d]_{h}&&{\mathbb C}P^3\ar@{->}[dll]^{T}\\&S^4&&}
\]
where $H_1\colon S^3\to{\mathbb C}P^1$ and $H_3\colon S^7\to{\mathbb C}P^3$ are the classical Hopf fibrations and $h\colon S^7\to S^4$ is the quaternionic Hopf fibration from Subsection~\ref{ssec:quatH}.


\section{Conclusions and future work}


In the paper, an interplay between two principal bundles is studied. It leads to a fiber bundle, called {\it nested bundle}, that is not principal in general. We described the relation between natural distributions, Hamiltonians, and geodesics on all three involved fiber bundles. The motivation for the study were some examples from geometry and physics. 

Similar systems can appear in the rolling problems, where, in a local chart, we can consider a subgroup $K$ of the group $H$ of isometric transformations acting on a configuration space of two rolling bodies. In the future we consider different rolling systems appearing in the robotics, or spline constructions related to the approximation of curves on Grassmann and/or Stiefel manifolds.


\end{document}